\documentclass[12pt]{amsart}
\usepackage{amscd}
\theoremstyle{plain}
\newtheorem{thm}{Theorem}
\newtheorem{cor}{Corollary}
\newtheorem{lem}{Lemma}
\newtheorem{prop}{Proposition}

\newcommand{\C}{\mathbb{C}}

\newcommand{\Z}{\mathbb{Z}}
\newcommand{\N}{\mathbb{N}}
\newcommand{\PP}{\mathbb{P}}
\newcommand{\OO}{\mathcal{O}}

\DeclareMathOperator{\GL}{GL}

\DeclareMathOperator{\tr}{tr}

\DeclareMathOperator{\Hom}{Hom}

\DeclareMathOperator{\Ker}{Ker}

\DeclareMathOperator{\Ima}{Im}

\DeclareMathOperator{\End}{End}

\DeclareMathOperator{\Rep}{Rep}

\DeclareMathOperator{\rank}{rank}
\DeclareMathOperator{\Res}{Res}

\begin{document}
\title[Monodromy for systems of vector bundles]{Monodromy for systems of vector bundles and multiplicative preprojective algebras}

\author{William Crawley-Boevey}
\address{Department of Pure Mathematics, University of Leeds, Leeds LS2 9JT, UK}
\email{w.crawley-boevey@leeds.ac.uk}


\begin{abstract}
We study systems involving vector bundles and logarithmic connections on Riemann surfaces
and linear algebra data linking their residues. This generalizes representations of
deformed preprojective algebras. Our main result is the
existence of a monodromy functor from such systems to
representations of a multiplicative preprojective algebra.
As a corollary, we prove that the multiplicative preprojective algebra
associated to a Dynkin quiver is isomorphic to the usual preprojective algebra.
\end{abstract}
\maketitle

\section{Introduction}
In this paper we consider systems of vector bundles and logarithmic connections on
Riemann surfaces, combined with a certain linking between them which involves
the residues of the connections. The configuration is described by specifying
both a quiver and a Riemann surface with finitely many connected components.

Our main result is the existence of a monodromy functor which, in case the 
Riemann surface is compact, gives an equivalence between a category of such 
systems and a category of representations of a multiplicative preprojective
algebra, in the sense of \cite{CBS}.

In previous work (see \cite{CBicm} for a survey) we used deformed preprojective algebras,
multiplicative preprojective algebras and monodromy to study questions about the
existence of tuples of square matrices in prescribed conjugacy classes whose sum is zero or
whose product is the identity. In that work the quiver was taken to be star-shaped,
and the restriction of a representation to each arm of the star fixed one of the matrices.
This paper generalizes some of that work, and provides a more natural setting,
as it enables one apply monodromy in a quiver setting, without having to pass to and 
from tuples of matrices.

As an application, we prove that the multiplicative preprojective algebra $\Lambda^1(Q)$
associated to a Dynkin quiver is isomorphic to the usual preprojective algebra $\Pi(Q)$,
answering a question of Shaw \cite{Shaw}.
We also give an interpretation of this as a special case of Hilbert's 21st problem.

The structure of the paper is as follows. In section~\ref{s:vbr} we introduce the notion 
of a `Riemann surface quiver', which combines a quiver with a Riemann surface, 
and we also introduce `vector bundle representations'. In section~\ref{s:conn} 
we introduce the idea of a `$\lambda$-connection system' on a vector bundle representation, 
combining a logarithmic connection with linear algebra data linking its residues. 
In section~\ref{s:lift} we prove a `Lifting Theorem', which gives
a criterion for the existence of $\lambda$-connection systems, 
and in section~\ref{s:monodromy} we discuss the `Monodromy Theorem', 
giving the existence of a monodromy functor;
the proof uses two lemmas, which are given in sections~\ref{s:cqe} and \ref{s:lcd}.
In section~\ref{s:mpa} we identify the target of the monodromy functor in the compact case 
with a category of representations of a multiplicative preprojective algebra.
Finally, we address Shaw's question and Hilbert's 21st problem in section~\ref{s:cor}.

I would like to thank R.~Bielawski and T.~Hausel for suggesting that the constructions
in this paper might be considered in the context of torsion-free sheaves on a singular curve.
I hope to return to this in another paper.

\section{Vector bundle representations}
\label{s:vbr}
By a \emph{Riemann surface quiver} $\Gamma$ we mean a quiver whose set of vertices has the
structure of a Riemann surface $X$, not necessarily connected, but with only finitely many connected components.
We assume in addition that $\Gamma$ has only finitely many arrows.
We say that $\Gamma$ is \emph{compact} if $X$ is compact, and that $\Gamma$ is of \emph{$\PP^1$-type}
if each connected component is a copy of the Riemann sphere $\PP^1$.

By a \emph{marked} point, we mean a point of the Riemann surface which occurs as a head or tail of an arrow.
We say that $\Gamma$ has \emph{non-interfering arrows} provided that each
point of the Riemann surface occurs at most once as the head or tail of an arrow,
so that if there are $n$ arrows, then there are exactly $2n$ marked points.

Let $X_i$ ($i\in I$) be the connected components of $X$, for a suitable finite indexing set $I$.
Given $p\in X$ we denote by $[p]$ the index of the connected component containing $p$.
We define the \emph{component quiver} $[\Gamma]$ of $\Gamma$ to be the quiver
with vertex set $I$ and with an arrow $[a]:[p]\to [q]$ for each arrow $a:p\to q$ in $\Gamma$.

By a \emph{vector bundle representation} of a Riemann surface quiver $\Gamma$
we mean a representation of the quiver such that the vector spaces at each vertex
form the fibres of a (holomorphic) vector bundle.
Equivalently it is a collection $\mathbf{E} = (E,E_a)$ consisting of a vector bundle $E$ over $X$
and linear maps  $E_a : E_p\to E_q$ for each arrow $a:p\to q$ in $\Gamma$. Here $E_p$ denotes the fibre of $E$ at $p$.
Note that specifying the vector bundle $E$ over $X$ is equivalent to specifying
a vector bundle $E_i$ on each $X_i$, and so we can also consider a vector bundle system
as a collection $(E_i,E_a)$.
We do not require the $E_i$ to all have the same rank, and define the \emph{dimension vector}
of $\mathbf{E}$ to be the vector in $\N^I$ whose $i$th component is $\rank E_i$.
If $\Gamma$ is compact, the \emph{degree} of $\mathbf{E}$ is $\deg E = \sum_{i\in I} \deg E_i$.

There is a natural category $\mathcal{V}_\Gamma$ of vector bundle representations of $\Gamma$, in which
a homomorphism $\theta$ from $\mathbf{E}=(E,E_a)$ to $\mathbf{E}'=(E',E'_a)$
is a vector bundle homomorphism $\theta:E\to E'$
(or equivalently a collection of vector bundle homomorphisms $\theta_i:E_i\to E'_i$)
with the property that $\theta_q E_a = E'_a \theta_p$ for each arrow $a:p\to q$.
This is clearly an additive category over $\C$ with split idempotents. If $\Gamma$ is
compact it has finite-dimensional homomorphism spaces.

We say that a vector bundle representation $\mathbf{E}$ is \emph{vb-trivial}
if the bundles $E_i$ are trivial vector bundles (of some rank) for all $i$.
If $\Gamma$ is compact then the category of trivial vector bundles
on $X_i$ is equivalent to the category of vector spaces,
and hence the category of vb-trivial vector bundle representations
of $\Gamma$ is equivalent to the category of representations of the component quiver.

\section{Connections}
\label{s:conn}
Henceforth we fix a Riemann surface quiver $\Gamma$ with underlying Riemann surface $X$.
Let $D$ be the set of marked points of $X$.
We fix a collection of scalars $\lambda = (\lambda_p)_{p\in D} \in \C^D$.
Let $D_i = D\cap X_i$, and define $\lambda_i = \sum_{p\in D_i} \lambda_p$ for $i\in I$.

By a \emph{$\lambda$-connection system} on
a vector bundle representation $\mathbf{E}$ of $\Gamma$,
we mean a collection $\pmb{\nabla} = (\nabla,\nabla_a)$
consisting of a connection $\nabla : E \to \Omega^1_X(\log D)\otimes E$ on $E$,
holomorphic except possibly for logarithmic poles on $D$
(or equivalently a connection $\nabla_i$ on each $E_i$),
and linear maps $\nabla_a:E_q\to E_p$ for each arrow $a:p\to q$ in $\Gamma$,
satisfying
\[
\Res_p \nabla - \lambda_p 1_{E_p} = \sum_{t(a)=p} \nabla_a E_a - \sum_{h(a)=p} E_a \nabla_a
\]
for points $p\in D$.
If the Riemann surface quiver has non-interfering arrows we can rewrite this as
\[
\Res_p \nabla - \lambda_p 1_{E_p} = \nabla_a E_a
\quad\text{and}\quad
\Res_q \nabla - \lambda_q 1_{E_q} = - E_a \nabla_a
\]
for each arrow $a:p\to q$ in $\Gamma$.

We denote by $\mathcal{C}_{\Gamma,\lambda}$ the category whose objects are pairs
$(\mathbf{E},\pmb{\nabla})$,
consisting of a vector bundle representation $\mathbf{E}$ of $\Gamma$
and a $\lambda$-connection system $\pmb{\nabla}$ on $\mathbf{E}$,
and in which a morphism from $(\mathbf{E},\pmb{\nabla})$ to $(\mathbf{E}',\pmb{\nabla}')$
is by definition a morphism of vector bundle representations $\theta:\mathbf{E}\to \mathbf{E}'$
with the property that the square
\[
\begin{CD}
E @>\nabla>> \Omega^1_X(\log D)\otimes E \\
@V\theta VV @VV 1\otimes \theta V \\
E' @>\nabla'>> \Omega^1_X(\log D)\otimes E'
\end{CD}
\]
commutes and $\theta_p \nabla_a = \nabla'_a \theta_q$ for each arrow $a:p\to q$ in $\Gamma$.
Note the following interpretation in terms of deformed preprojective algebras \cite{CBH}.

\begin{prop}
\label{p:dpa}
If $\Gamma$ is of $\PP^1$-type, then
the full subcategory of $\mathcal{C}_{\Gamma,\lambda}$
given by pairs $(\mathbf{E},\pmb{\nabla})$ with $\mathbf{E}$ vb-trivial
is equivalent to the category of representations of $\Pi^\lambda([\Gamma])$,
the deformed preprojective algebra associated to the component quiver $[\Gamma]$ of $\Gamma$.
\end{prop}

\begin{proof}
As the $E_i$ are trivial vector bundles, all fibres $E_p$ ($p\in X_i$) can all be identified,
say with vector space $V_i$, and then these spaces together with the linear maps $E_a$ and $\nabla_a$
define a representation of the double quiver of $[\Gamma]$.

Any logarithmic connection on $E_i$ has the form
\[
\nabla_i(f) = df + A_i(z) \cdot f \, dz
\]
for $f$ is a section of $E_i$, where $A_i(z)$ is an $\End(V_i)$-valued function of $z$, holomorphic
except for simple poles on $D_i$.
If $z$ is a coordinate for $X_i$ with $\infty\notin D_i$, then $A_i(z)$ must have the form
\[
A_i(z) = \sum_{p\in D_i} \frac{R_p}{z-p}
\]
where the residues $R_p = \Res_p \nabla_i \in \End(V_i)$
are any endomorphisms, subject to the relation $\sum_{p\in D_i} R_p = 0$;
see for example \cite[\S 1.2]{AB}.
Using the linking of residues, this relation can be rewritten as
\[
\sum_{h(a)\in X_i} E_a \nabla_a - \sum_{t(a)\in X_i} \nabla_a E_a = \lambda_i 1_{V_i}.
\]
This is the relation at vertex $i$ for the deformed preprojective algebra. The assertion follows.
\end{proof}

\section{Lifting Theorem}
\label{s:lift}
The following is a straightforward analogue of \cite[Theorem 3.3]{CBmoment} and \cite[Theorem 7.1]{CBipb}.

\begin{thm}
\label{t:lift}
If $\mathbf{E}$ is a vector bundle representation of a compact Riemann surface quiver $\Gamma$, then
there exists a $\lambda$-connection system on $\mathbf{E}$ if and only if
\[
\deg E' + \sum_{i\in I} \lambda_i \rank E'_i = 0
\]
for every direct summand $\mathbf{E}'$ of $\mathbf{E}$.
\end{thm}

\begin{proof}
We identify $f\in \End(E)$ with $(f_i)\in \bigoplus_{i\in I}\End(E_i)$.
Since $X_i$ is compact, there is a trace mapping $\tr_i:\End(E)\to \C$ with $\tr(f) = \tr(f_p)$ for all $p\in X_i$.

A criterion of Mihai \cite{Mihai2} for the existence of logarithmic connections, see \cite[Theorem 7.2]{CBipb},
implies that there is a $\lambda$-connection system on $\mathbf{E}$
if and only if there are linear maps $\nabla_a:E_q\to E_p$ for all arrows $a:p\to q$ in $\Gamma$ satisfying
\[
\sum_{p\in D} \tr\biggl( (\lambda_p 1 + \sum_{t(a)=p} \nabla_a E_a - \sum_{h(a)=p} E_a \nabla_a) f_p \biggr)
= \frac{1}{2\pi\sqrt{-1}} \; \langle b(E),f\rangle
\]
for all $f\in \End(E)$, where $\langle b(E),f\rangle = \sum_{i\in I} \langle b(E_i),f_i\rangle$.

There is an exact sequence
\[
0\to \End(\mathbf{E}) \to \End(E) \to \bigoplus_{a:p\to q} \Hom(E_p,E_q)
\]
where the second map sends $f$ to the tuple whose $a$th component is $E_a f_p - f_q E_a$.
Using the trace pairing it dualizes to give an exact sequence
\[
\bigoplus_{a:p\to q} \Hom(E_q,E_p) \xrightarrow{F} \End(E)^* \xrightarrow{G} \End(\mathbf{E})^* \to 0
\]
where $F$ sends a tuple $(\nabla_a)$ to the linear form sending $f$ to
\[
\sum_{a:p\to q} \tr \biggl( (E_a f_p - f_q E_a)\nabla_a \biggr)
= \sum_{p\in D} \tr \biggl( (\sum_{t(a)=p}\nabla_a E_a - \sum_{h(a)=p} E_a\nabla_a) \; f_p \biggr).
\]
If we define $\xi\in\End(E)^*$ by
\[
\xi(f) = \frac{1}{2\pi\sqrt{-1}} \; \langle b(E),f\rangle - \sum_{i\in I}\sum_{p\in D_i} \lambda_i \tr_i(f)
\]
then the condition for there to be a $\lambda$-connection system can be written as $\xi\in\Ima(F)$,
so as $\xi\in\Ker(G)$, so as $\xi(f) = 0$ for all $f\in \End(\mathbf{E})$.
Now if $f=p_{\mathbf{E}'}$, the projection onto a direct summand $\mathbf{E}'$ of $\mathbf{E}$, then
\[
\xi(p_{\mathbf{E}'}) = \sum_{i\in I} \left( - \deg E'_i - \sum_{p\in X_i} \lambda_i \rank E'_i \right),
\]
so the existence of a $\lambda$-connection system implies the condition in the statement of the Theorem.
Conversely, since $\xi(f) = 0$ for $f$ nilpotent, the stated condition implies the existence of
a $\lambda$-connection system on any indecomposable direct summand of $\mathbf{E}$, and by combining these
one obtains a $\lambda$-connection system on $\mathbf{E}$.
\end{proof}

\section{Monodromy Theorem}
\label{s:monodromy}
Let $\Gamma$ be a Riemann surface quiver with non-interfering arrows, and let $\lambda\in \C^D$, as before.
We fix a subset $T$ of $\C$.
We assume that $T$ is \emph{non-resonant},
by which we mean that distinct elements of $T$ never differ by an integer,
and we assume moreover that $0\in T$.


We say that a $\lambda$-connection system $\pmb{\nabla}$ \emph{has eigenvalues in $T$} provided that
$E_a \nabla_a$ has eigenvalues in $T$ for all arrows $a$ in $\Gamma$.
(Since $0\in T$, it is equivalent that $\nabla_a E_a$ has eigenvalues in $T$.)
We denote by $\mathcal{C}^T_{\Gamma,\lambda}$ the full subcategory of $\mathcal{C}_{\Gamma,\lambda}$
given by the pairs $(\mathbf{E},\pmb{\nabla})$ such that $\pmb{\nabla}$ has eigenvalues in $T$.

For each $i\in I$, fix a base point $b_i$ for $X_i \setminus D_i$, and for each $p\in D_i$
fix a loop $\ell_p \in \pi_1(X_i\setminus D_i,b_i)$ around $p$.
Define $\sigma \in (\C^\times)^D$ by $\sigma_p = e^{2\pi \sqrt{-1}\; \lambda_p}$ for $p\in D$.

We denote by $\Rep_\sigma \pi(\Gamma)$ the category whose
objects are given by collections $(V_i,\rho_i,\rho_a,\rho_a^*)$ consisting of
a representation $\rho_i(\pi_1(X_i\setminus D_i,b_i))\to \GL(V_i)$
for each $i\in I$,
and linear maps $\rho_a:V_i\to V_j$ and $\rho_a^*:V_j\to V_i$
for each arrow $a:p\to q$ in $\Gamma$, where $i = [p]$ and $j = [q]$,
satisfying
\[
\sigma_p^{-1} \rho_i(\ell_p)^{-1} =  1_{V_i} + \rho_a^* \rho_a
\quad\text{and}\quad
\sigma_q \rho_j(\ell_q) = 1_{V_j} + \rho_a \rho_a^*
\]
and whose morphisms are the natural ones.

Let $S = \{ e^{2\pi \sqrt{-1} \; t} - 1 : t\in T\}$.
We denote by $\Rep_\sigma^S \pi(\Gamma)$ the full subcategory
of $\Rep_\sigma \pi(\Gamma)$ consisting of the collections
in which $\rho_a \rho_a^*$ has eigenvalues in $S$ for all arrows $a$.
(Since $0\in S$, it is equivalent that $\rho_a^* \rho_a$ has eigenvalues in $S$.)

\begin{thm}
\label{t:monthm}
There is an equivalence of categories
$\mathcal{C}^T_{\Gamma,\lambda}\to \Rep_\sigma^S \pi(\Gamma)$.
\end{thm}

Note the following special cases: if $T=\{0\}$, the theorem gives an equivalence
$\mathcal{C}^{\{0\}}_{\Gamma,\lambda}\to \Rep_\sigma^{\{0\}} \pi(\Gamma)$;
if $T$ is a transversal to $\Z$ in $\C$ such as $\{ z\in \C : 0\le \Re z< 1\}$, it
gives an equivalence $\mathcal{C}^T_{\Gamma,\lambda}\to \Rep_\sigma \pi(\Gamma)$.
Note also that one could specify a different set $T$ for each arrow $a$ in $\Gamma$.

\begin{proof}
Deleting all arrows and marked points, monodromy gives an equivalence
from the category consisting of vector bundles on $X\setminus D$ equipped with a holomorphic connection,
to the category whose objects are a collection of representations $(\rho_i)$ of the fundamental
groups $\pi_1(X_i\setminus D_i,b_i)$, see for example \cite[Theorem 1.1]{Malg}.

We need to check that the ways of extending the vector bundles and connections to give vector bundle
representations and $\lambda$-connection systems on $\Gamma$ correspond to the ways of extending the
representations of the fundamental groups to give objects in $\Rep_\sigma^S \pi(\Gamma)$.

In fact it suffices to check this locally for each arrow $a:p\to q$ in $\Gamma$.
Thus, replacing $X$ by the union of sufficiently small disk neighbourhoods of $p$ and $q$, we may suppose
that $X$ is the disconnected union of two disks $X_1 \cup X_2$, that $p$ and $q$ are the centres of these disks,
and that $a:p\to q$ is the only arrow in $\Gamma$.

Now using Lemma~\ref{l:lcd}, we see that
$\mathcal{C}^T_{\Gamma,\lambda}$ can be considered as the category
of pairs of vector spaces $E_p$ and $E_q$
equipped with linear maps $R_p:E_p\to E_p$, $R_q:E_q \to E_q$,
$E_a:E_p\to E_q$, $\nabla_a:E_q\to E_p$ satisfying
\[
R_p - \lambda_p 1 = \nabla_a E_a
\quad\text{and}\quad
R_q - \lambda_q 1 = - E_a \nabla_a,
\]
and with $\nabla_a E_a$ having eigenvalues in $T$.

On the other hand, $\ell_p$ and $\ell_q$ freely generate the fundamental groups of $X_1\setminus \{p\}$ and $X_2\setminus\{q\}$,
so an object in $\Rep_\sigma^S \pi(\Gamma)$ is given by
vector spaces $V_1$, $V_2$, elements $\rho_1(\ell_p)\in\GL(V_1)$ and $\rho_2(\ell_q)\in\GL(V_2)$
and linear maps $\rho_a:V_1\to V_2$ and $\rho_a^*:V_2\to V_1$ satisfying
and
\[
\sigma_p^{-1} \rho_1(\ell_p)^{-1} = 1 + \rho_a^* \rho_a
\quad\text{and}\quad
\sigma_q \rho_2(\ell_q) = 1 + \rho_a \rho_a^*
\]
and such that the eigenvalues of $\rho_a\rho_a^*$ are in $S$.

The equivalence between these categories is given by the case $m=2$ of Lemma~\ref{l:cqe},
sending the object given by $E_p$, $E_q$, $R_p$, $R_q$, $E_a$, $\nabla_q$ to the object with
$V_1 = E_p$, $V_2 = E_q$, $\rho_a = E_a$ and
\[
\rho_a^* = \sum_{n=1}^\infty \frac{(2\pi \sqrt{-1})^n}{n!} \; \nabla_a (E_a \nabla_a )^{n-1}.
\]
\end{proof}

\section{Cyclic quiver exponential}
\label{s:cqe}
Let $Q_m$ be the cyclic quiver with vertices $1,\dots,m$ and arrows $a_1:m\to 1$ and $a_i:i-1\to i$ ($i=2,\dots,m$).
Given a subset $\Sigma\subseteq \C$ with $0\in \Sigma$, we denote by $\Rep^\Sigma Q_m$
the category of representations of $Q_m$
in which the path $a_m \dots a_1$ (and hence, since $0\in \Sigma$, any path of length $m$)
is represented by an endomorphism with eigenvalues in $\Sigma$.
Let $T$ be a non-resonant set with $0\in T$ and let
$S = \{ e^{2\pi \sqrt{-1} \; t} - 1 : t\in T\}$.

\begin{lem}
\label{l:cqe}
There is an equivalence $\Rep^T Q_m\to \Rep^S Q_m$ sending a representation
of $Q_m$ by vector spaces $V_i$ and linear maps $a_i:V_{i-1}\to V_i$ to the representation
with the same vector spaces, and linear maps
\[
a_1' = \sum_{n=1}^\infty \frac{(2\pi \sqrt{-1})^n}{n!} \; a_1 (a_n \dots a_1)^{n-1}
\]
and $a_i' = a_i$ for $i\neq 1$.
\end{lem}

\begin{proof}
The functor $F$ given by this construction, and which acts as the identity on morphisms,
defines an equivalence from the category of nilpotent representations to itself,
as an inverse is given by the analogous logarithm functor, with
\[
a_1'' = \frac{1}{2\pi\sqrt{-1}} \sum_{n=1}^\infty \frac{(-1)^{n-1}}{n} \; a_1 (a_n \dots a_1)^{n-1}.
\]
Given an endomorphism $\sigma$ of a vector space $V$, let $B(\sigma)$ be the representation
of $Q_m$ in which all $V_i=V$, $a_1=\sigma$ and $a_i=1$ for $i\neq 1$.
Recall that any indecomposable representation of $Q_m$ is either nilpotent or
isomorphic to $B(\sigma)$ for some indecomposable automorphism $\sigma$.
Now if $\sigma$ is given by a Jordan block $J_n(\lambda)$, then
\[
F(B(\sigma)) = B(e^{2\pi\sqrt{-1}\;\sigma}-1) \cong B(J_n(e^{2\pi\sqrt{-1}\; \lambda}-1)).
\]
It follows that $F$ induces a surjective map onto the isomorphism classes of $\Rep^S Q_m$.
It is a faithful functor, and also full since the non-resonance of $T$ ensures that
\[
\dim \Hom(F(X),F(Y)) = \dim \Hom(X,Y)
\]
for any $X,Y$ nilpotent or of the form $B(\sigma)$.
\end{proof}

\section{Logarithmic connections for a disk}
\label{s:lcd}
The result in this section is well known, but we state it carefully
as the literature can be confusing.

Let $X$ an open unit disk in $\C$ with centre $p$, let $D = \{p\}$,
and let $\Sigma$ be a non-resonant subset of $\C$.
Let $\mathcal{C}^\Sigma$ be the category of pairs $(E,\nabla)$
consisting of a vector bundle $E$ on $X$ and a logarithmic connection
$\nabla:E\to\Omega^1_X(\log D)\otimes E$, such that the eigenvalues of $\Res_p\nabla$ are in $\Sigma$.
A representation of the polynomial ring $\C[x]$ can be thought of as a pair $(V,R)$ where $V$ is a vector
space and $R\in \End(V)$. We denote by $\Rep^\Sigma \C[x]$ be the category of representations such that
$R$ has eigenvalues in $\Sigma$.
Given $(V,R)$, one obtains a trivial vector bundle $E = \OO\otimes_C V$ and a
logarithmic connection $\nabla_R : E \to \Omega^1_X(\log D)\otimes E$ defined by
\[
\nabla_R(f) = df + \frac{R}{z-p} \cdot f \, dz
\]
for $f$ a section of $E$. It has residue $\Res_p \nabla_R = R$.

\begin{lem}
\label{l:lcd}
The functor $F(V,R) = (\OO\otimes V,\nabla_R)$ induces
an equivalence $\Rep^\Sigma \C[x]\to \mathcal{C}^\Sigma$.
\end{lem}

\begin{proof}
By \cite[Theorem 6.1]{CBipb}, monodromy gives an equivalence $M$ from $\mathcal{C}^\Sigma$ to $\Rep^H \pi_1(X\setminus D)$,
the category of representations of the fundamental group of $X\setminus D$ in which the loop around $p$ is given by
a matrix with eigenvalues in $H = \{ e^{-2\pi \sqrt{-1} \; \sigma} : \sigma\in \Sigma\}$.

The composition $MF : \Rep^\Sigma \C[x]\to \Rep^H \pi_1(X\setminus D)$ sends $(V,R)$ to the representation with
vector space $V$ and in which the loop around $p$ is given by $e^{-2\pi \sqrt{-1}\; R}$,
so by properties of the matrix exponential (as in the case $m=1$ of Lemma~\ref{l:cqe}) it is an equivalence.
It follows that $F$ is an equivalence.
\end{proof}

\section{Multiplicative preprojective algebras}
\label{s:mpa}
In case the Riemann surface quiver $\Gamma$ is compact,
we show that the target of the monodromy functor can be identified with
a category of representations of a multiplicative preprojective algebra \cite{CBH}.

Let $Q$ be the quiver obtained from the component quiver $[\Gamma]$ of $\Gamma$
by adjoining $g_i$ loops $\tau_i^j$ ($1 \le j \le g_i$) at each vertex $i$, where $g_i$ is the genus of $X_i$.
Define $q = (q_i)\in (\C^\times)^I$ by $q_i = e^{2\pi \sqrt{-1} \; \lambda_i}$ for $i\in I$,
and let $\Lambda^q(Q)$ be the corresponding multiplicative preprojective algebra.
We denote by $\Rep \Lambda^q(Q)'$  the category of representations of $\Lambda^q(Q)$ in which the linear maps
representing the loops $\tau_i^j$ are invertible. It can also be defined as the category
of representations of a universal localization of $\Lambda^q(Q)$.

\begin{prop}
\label{p:mpa}
$\Rep_\sigma \pi(\Gamma)$ is equivalent to $\Rep \Lambda^q(Q)'$.
\end{prop}

\begin{proof}
If $a_1,\dots,a_r$ are the arrows with head in $D_i$ and $b_1,\dots,b_s$ are the arrows with tail in $D_i$,
then $\pi(X_i\setminus D_i)$ can be generated by elements
$u_1$, \dots, $u_{g_i}$, $v_1$, \dots, $v_{g_i}$, $\ell_1$, \dots, $\ell_{r+s}$, subject to the relation
\[
u_1 v_1 u_1^{-1} v_1^{-1} \cdots
u_{g_i} v_{g_i} u_{g_i}^{-1} v_{g_i}^{-1}
\ell_1 \cdots \ell_{r+s} = 1
\]
where $\ell_1,\dots,\ell_r$ are loops in $X_i\setminus D_i$ around the heads of $a_1,\dots,a_r$ and
$\ell_{r+1},\dots,\ell_{r+s}$ are loops around the tails of $b_1,\dots,b_s$.

A representation $\rho:\pi(X_i\setminus D_i)\to\GL(V_i)$ is determined by
automorphisms $\rho(u_j)$, $\rho(v_j)$, $\rho(\ell_j)$
of $V_i$ satisfying the same relation.
Rewriting $\rho(u_j)$ and $\rho(v_j)$ in terms of endomorphisms $e_j$ and $e_j^*$ of $V_i$ via
\[
\rho(u_j) = e_j
\quad
\text{and}
\quad
\rho(v_j) = e_j^{-1} + e_j^*,
\]
and using the equations
\[
\sigma_p^{-1} \rho_i(\ell_p)^{-1} =  1_{V_i} + \rho_a^* \rho_a
\quad\text{and}\quad
\sigma_q \rho_j(\ell_q) = 1_{V_j} + \rho_a \rho_a^*,
\]
the relation becomes
\begin{multline*}
(1+ e_1 e_1^*)
(1+ e_1^* e_1)^{-1}
\cdots
(1+ e_{g_i} e_{g_i}^*)
(1+ e_{g_i}^* e_{g_i})^{-1}
\cdot
\\
\cdot
(1 + \rho_{a_1}\rho_{a_1}^*)
\cdots
(1 + \rho_{a_r}\rho_{a_r}^*)
(1 + \rho_{b_1}^*\rho_{b_1})^{-1}
\cdots
(1 + \rho_{b_s}^*\rho_{b_s})^{-1}
= q_i 1
\end{multline*}
This is the relation at vertex $i$ for the multiplicative preprojective algebra $\Lambda^q(Q)$,
where $e_j$ and $e_j^*$ represent $\tau_i^j$ and $\tau_i^{j*}$.
Observe that the the requirement for the multiplicative preprojective algebra
that all terms in this product be invertible,
coupled with the requirement that the $\tau_i^j$ be represented by invertible linear maps,
corresponds to the fact that the generators of the fundamental group
are given in the representation $\rho$ by invertible linear maps.
\end{proof}

\section{Corollaries}
\label{s:cor}
We obtain the following answer to a question of Shaw \cite[\S 5.2]{Shaw}.
Note that if one works not over $\C$, but over a base field of characteristic 2, then the assertion is false, see \cite[Lemma 5.5.1]{Shaw}

\begin{cor}
If $Q$ is a quiver of Dynkin type $ADE$, then $\Lambda^1(Q)$ is isomorphic to the usual preprojective algebra $\Pi(Q)$.
\end{cor}

\begin{proof}
Let $\Gamma$ be a Riemann surface quiver of $\PP^1$-type whose component quiver is $Q$.
The Monodromy Theorem gives an equivalence
$\mathcal{C}^T_{\Gamma,0} \to \Rep \Lambda^q(Q)$
for a suitable transversal $T$.

By \cite[Lemma 5.1]{CBS}, the simple representations for $\Lambda^1(Q)$ are 1-dimensional at a vertex.
Thus they come from vector bundle representations and $0$-connection systems which consist of a
single line bundle equipped with a holomorphic connection.
Moreover the line bundles have degree 0 (for example by the Lifting Theorem), so,
since $\Gamma$ is of $\PP^1$-type, they are trivial.

Now any representation of $\Lambda^1(Q)$ is an iterated extension of simple representations,
so any object of $\mathcal{C}^T_{\Gamma,0}$ is an iterated extension of trivial line bundle systems.
Since the trivial line bundle on $\PP^1$ has no self-extensions, it follows that
any object of $\mathcal{C}^T_{\Gamma,0}$ is vb-trivial.
Thus by Proposition~\ref{p:dpa}, the objects of $\mathcal{C}^T_{\Gamma,0}$
can be considered as representations of $\Pi(Q)$.
Now as $Q$ is Dynkin, $\Pi(Q)$ is finite-dimensional and any representation is nilpotent.
It follows that $\mathcal{C}^T_{\Gamma,0}$ is equivalent to $\Rep \Pi(Q)$.

Thus the Monodromy Theorem gives an equivalence $\Rep \Pi(Q)\to \Rep \Lambda^1(Q)$.
Here we have only considered finite-dimensional representations,
but $\Lambda^1(Q)$ is finite-dimensional by \cite[3.1.1]{Shaw}, so this equivalence is a Morita equivalence.
Finally, as the simple representations are 1-dimensional for both algebras, the algebras are isomorphic.
\end{proof}

This gives the following simple case of Hilbert's 21st problem (which is presumably already known).

\begin{cor}
Given $n_1,n_2,n_3$ with
$1/n_1 + 1/n_2 + 1/n_3 > 1$ and $D=\{a_1,a_2,a_3\} \subset \PP^1$,
any representation
$\rho : \pi_1(\PP^1 \setminus D) \to \GL_n(\C)$
sending loops around the $a_i$ to
unipotent matrices $\rho_i$ with \makebox{$(\rho_i-1)^{n_i}=0$}
arises as the monodromy of a logarithmic connection on a trivial vector bundle (so a Fuchsian system).
\end{cor}

\begin{proof}
We may assume that $\rho_1\rho_2\rho_3 = 1$.
The representation corresponds, as in \cite[\S 8]{CBS}, to a representation of
$\Lambda^1(Q)$ for some star-shaped quiver $Q$.
The condition on $n_i$ ensures that $Q$ is of Dynkin type.
By the previous result it arises from some representation of $\Pi(Q)$, which can be
interpreted as a vb-trivial object in $\mathcal{C}^T_{\Gamma,0}$ as in the previous proof. The connection
on the trivial vector bundle at the central vertex of $Q$ is the required Fuchsian system.
\end{proof}


\begin{thebibliography}{99}
\bibitem{AB}
D. V. Anosov, A. A. Bolibruch,
The Riemann-Hilbert problem,
Friedr. Vieweg \& Sohn, Braunschweig, 1994.

\bibitem{CBmoment}
W. Crawley-Boevey,
Geometry of the moment map for representations of quivers,
Compositio Math. 126 (2001), 257--293.

\bibitem{CBipb}
W. Crawley-Boevey,
Indecomposable parabolic bundles and the existence of matrices in
prescribed conjugacy class closures with product equal to the identity,
Publ. Math. Inst. Hautes \'Etudes Sci. 100 (2004) 171--207.

\bibitem{CBicm}
W. Crawley-Boevey,
Quiver algebras, weighted projective lines, and the Deligne-Simpson problem,
in: Sanz-Sol\'e et al. (Eds.), International Congress of Mathematicians, vol. 2 (Madrid, 2006),
European Mathematical Society, 2007, pp. 117--129.

\bibitem{CBH}
W. Crawley-Boevey, M. P. Holland,
Noncommutative deformations of Kleinian singularities,
Duke Math. J. 92 (1998) 605--635.

\bibitem{CBS}
W. Crawley-Boevey, P. Shaw,
Multiplicative preprojective algebras, middle convolution and the Deligne-Simpson problem,
Adv. Math. 201 (2006), 180--208.

\bibitem{Malg}
B. Malgrange,
Regular connections, after Deligne,
in: A. Borel et al., Algebraic D-modules, Academic Press, 1987, pp.129--149.


\bibitem{Mihai2}
A. Mihai,
Sur les connexions m\'eromorphes,
Rev.\ Roum.\ Math.\ Pures et Appl. 23 (1978) 215--232.

\bibitem{Shaw}
P. Shaw,
Generalisations of Preprojective algebras,
Ph. D. thesis, University of Leeds (2005).

\end{thebibliography}
\end{document}